\documentclass[11pt]{amsart}
\usepackage{amsmath,amssymb,amsthm}
\usepackage{graphicx}

\newcommand{\rp}{$\mathbb{RP}^2$}
\newcommand{\pg}{$\operatorname{PGL}(3,\mathbb{R})$}
\newcommand{\g}{$\operatorname{GL}(3,\mathbb{R})$}
\newcommand{\R}{$\mathbb{R}^3$}

\newtheorem{definition}{Definition}

\newtheorem{proposition}[definition]{Proposition}
\newtheorem{theorem}[definition]{Theorem}

\title{The Hilbert area of inscribed triangles and quadrilaterals}
\author{Scott A. Wolpert}
\address{Department of Mathematics\\University of Maryland\\College Park, MD 20742 USA}
\email{saw@math.umd.edu}
\date{\today}

\begin{document}

\begin{abstract}
Hilbert volume is an invariant of real projective geometry.  Polygons inscribed in polygons are considered for the real projective plane.  In two-dimensions the correspondence between Fock-Goncharov and Cartesian coordinates is examined.  Degeneration and Hilbert area of inscribed quadrilaterals are analyzed.  A microlocal condition is developed for bounded Hilbert area under degeneration.  The condition is applied to give a sequence of strictly convex domains with bounded Hilbert area and divergent Goldman parameters.     
\end{abstract}

\maketitle

\section{Introduction.}

A convex real projective structure for a closed surface $S$ of genus at least $2$ is given by a properly discontinuous action of a subgroup of \pg\ on a strictly convex domain in \rp.   The deformation space of marked projective structures on the surface $S$ is a component of the space of representations of $\pi_1(S)$ into \pg\ modulo \pg\  conjugation, the \pg\  representation variety.  A proper convex domain in \rp\ is equipped with a natural Finsler metric, the Hilbert metric, derived from the cross ratio on $\mathbb{RP}^1$.  The Hilbert metric and area form are invariant under the \pg\ action on a domain.  For the special case of the interior of a conic, the Hilbert metric is the hyperbolic metric, presented in the Beltrami-Klein model.  In general a convex real projective surface has elementary geometric invariants, including the systole, diameter, length spectrum and area. A Finsler metric also provides for a geodesic flow and topological entropy - the exponential growth rate of orbit counts in increasing metric balls in the universal cover.  Colbois-Vernicos-Verovic \cite{CVV2} showed that the Hilbert metric for a bounded convex domain  in $\mathbb{R}^n$ is Gromov hyperbolic if and only if there is a bound for the Hilbert area of ideal triangles for the domain.  

In seminal works Goldman \cite{Gdprj} and later Fock-Goncharov \cite{FGmcps} introduced global coordinates for the deformation space of marked convex real projective structures on a surface.  The deformation space is investigated by considering the large-scale behavior of geometric invariants.  Nie considered the  one-dimensional families of orbifolds defined by reflections in the faces of projective simplices \cite{Nie}.  He showed that in the deformation families the diameters tend to infinity and topological entropies tend to zero.   He further showed that the domains of discontinuity for the hyperbolic Coxeter groups limit in the Gromov-Hausdorff topology to their defining simplices.  Zhang considered strictly convex real projective structures for closed surfaces and special deformations - diverging Goldman internal pants parameters \cite{Zhgdeg}.  He found that systoles tend to infinity and topological entropies tend to zero. Foulon-Kim considered deforming structures by taking Goldman's bulging parameters to infinity \cite{FoKi}.  They found that Hilbert area tends to infinity and that Gromov-Hausdorff limits include Euclidean half cylinders.  Foulon-Kim found for Goldman's bulging deformation on a pants decomposition of the surface that the topological entropy limits to its maximum value for a pair of pants.  More generally Kim considered the space $\mathcal{A}_2$ of Anosov representations for the free group on two generators \cite{Kideg}. He described explicitly pants and tori degenerations in the boundary of $\mathcal{A}_2$.  Sun considered the space of unmarked strictly convex representations for surfaces with boundaries \cite{SunZ}. Sun in the manner of Mirzakhani, presented bounds for the Goldman symplectic volume of the subset of unmarked representations with bounded invariants.  

We are interested in the behavior of the Hilbert area in the elementary setting of convex polygons inscribed in convex polygons. The Hilbert area of convex polygons inscribed in convex polygons can be used to estimate the Hilbert area of convex structures for surfaces.  Colbois-Venicos-Verovic \cite{CVV} gave bounds for the Hilbert area form for a rectangle and used geometric arguments to provide lower and upper bounds for the area of ideal triangles.  They found that all ideal triangles within a domain have the same area if and only if the domain is an ellipsoid. Adeboye-Cooper \cite{AdCo} derived an exact formula for a modified-definition Hilbert area of a triangle and also derived the  lower and upper bounds for the area of ideal triangles.  They applied their formula to study the area of ideal triangles in properly convex domains.

Beginning with the work of Fock-Goncharov \cite{FGmcps}, we consider inscribed polygons in the real projective plane \rp. We explore the behavior of the Fock-Goncharov parameterization including analyzing degeneration and Hilbert area.  The analysis involves the relation between the Fock-Goncharov triangle triple ratio and edge parameters and  Cartesian coordinates for polygon vertices.  To bound Hilbert area we use the comparison principle, the elementary bounds for rectangles and coverings by overlapping rectangles.  We begin with the known case of inscribed triangles.  Then we apply our approach to develop the main result Theorem \ref{main}, on the behavior of Hilbert area for inscribed quadrilaterals.  The real projective linear group \pg\ acts on \rp\ and can be used to normalize four points.  Accordingly the moduli for a quadrilateral with specified triple ratio invariants is the location of the fourth vertex within an affine triangle. Representing the moduli triangle with vertices $(0,0), (1,1)$ and $(0,2)$, and a specified flag lines intersection (see Figure 2), we show in Theorem \ref{main} that the Hilbert area approaches infinity as the fourth vertex approaches the triangle boundary with a possible exception at $(0,0)$ and $(0,2)$. In particular if $(x,y)$ approaches $(0,0)$ with $y/x$ approaching $1$, then the Hilbert area is comparable to $1+((y/x)-1)\log 1/x$; with a corresponding condition at $(0,2)$. Diverging Hilbert area corresponds to $x$ approaching $0$ exponentially faster than the slope $y/x$ approaches unity.  The condition is microlocal - stated in terms of point and slope.  Analyzing inscribed vertices converging to a shallow angle vertex is the key step in establishing the theorem. The analysis is provided in Proposition 5 - analyzing two vertices of an inner polygon approaching a vertex of an outer polygon, \emph{while} the angle at the outer vertex is becoming straight (tending to $\pi$).   In a concluding discussion we consider the relation between the Fock-Goncharov parameters and the Goldman twist-bulge parameters.  We explore Kim's postulation that bounding Hilbert area bounds bulging parameters.   An example sequence is given of strictly convex domains with bounded Hilbert area and divergent bulging parameters.   

I would like to thank William Goldman, Inkang Kim and especially Zhe Sun for many conversations and suggestions.   

\section{Preliminaries.}
We review the expositions \cite{CTT,FGmcps}.   
The real projective plane \rp\ is the space of lines through the origin in \R. Scalar multiplication defines an action of $\mathbb{R}^{\times}$ on $\mathbb{R}^3-\{0\}$; \rp\ is the resulting quotient space.  A point (resp. line) in \rp\ corresponds to a line (resp. plane) through the origin in \R.  Writing column vectors for the points of \R, a point in the projective plane is an equivalence class $\{\lambda(\alpha,\beta,\gamma)^{\top}\mid\lambda\in\mathbb{R}^{\times}\}$ and a line in the projective plane is an equivalence class $\{(a,b,c)\cdot(\alpha,\beta,\gamma)=0)\}$ for $(a,b,c)$ a vector in the dual of \R.  The linear group
\g\ acts on \R\ and the projective linear group \pg\ acts on the quotient \rp.   A set of points in \rp\ is collinear if there is a line containing the points.  A set of lines is concurrent if the lines have a point in common.  A set is convex if its intersection with each line is connected.   Four points (resp. four lines) are in general position if no three are collinear (resp. incident on  a point).  The action of \pg\ is simply transitive on ordered $4$-tuples of points or lines in general position.   Given a triple of lines in general position, an element of \g\ transforms the planes to coordinate planes.  Accordingly a pair of general position lines intersects in a single point and separates \rp\ into two convex sectors.   Accordingly a triple of general position lines intersects in three points and separates \rp\ into four convex triangular regions. Following Fock-Goncharov we denote points (resp. lines) by upper case (resp. lower case) letters.   A {\it flag} for \rp\ consists of a point on a line.   We use corresponding letters to denote the components of a flag; $A$ is the point on the line $a$ and may refer to flag by its line.   We say that $ABC$ is inscribed in $abc$ provided $ABC$ is inscribed in one of the four triangles.   We will study inscribed triangles and quadrilaterals.  

We consider the cross ratio and triple ratio.   The {\it cross ratio} of four points $x_1,x_2,x_3,x_4$ on a line is the value at $x_4$ of a projective coordinate for the line taking value $\infty$ at $x_1$, $-1$ at $x_2$ and $0$ at $x_3$.   Equivalently for an affine coordinate on the line the cross ratio is $\frac{(x_1-x_2)(x_3-x_4)}{(x_1-x_4)(x_2-x_3)}$.  Similarly the set of lines incident on a point form a projective line.  The cross ratio of an ordered quadruple of incident lines is defined.   Equivalently the cross ratio of an ordered tuple of incident lines is  determined by the intersection points with a general reference line.   The triple ratio of a labeled triple of flags $AaBbCc$ is defined as follows.  Choose linear functionals $f_a,f_b,f_c$ representing the projective lines and vectors $\tilde A, \tilde B, \tilde C$ representing the projective points. The {\it triple ratio} is defined as 
\begin{equation}
\label{triple}
\frac{f_a(\tilde B)f_b(\tilde C)f_c(\tilde A)}{f_a(\tilde C)f_b(\tilde A)f_c(\tilde B)}.
\end{equation}
The triple ratio is independent of the choices and depends only on the cyclic ordering of the flags.  Fock-Goncharov give alternate descriptions of the triple ratio.  

We consider the Hilbert metric and area.  Suppose $\Omega$ is a convex set contained in an affine set with Euclidean norm $|\cdot|$ and distance $d(\cdot,\cdot)$.  For distinct points $x,y$, let $p,q$ be the intersection points of the line $xy$ with $\partial\Omega$ such that $p,x,y,q$ are in order, then the Hilbert distance is
$$
d_{\Omega}(x,y)\,=\,\frac12 \log\frac{|p-y||q-x|}{|p-x||q-y|}.
$$
The distance has an infinitesimal form.  For $(x,v)$ a tangent vector at $x$ then
$$
\|v\|\,=\,\frac12\big(\frac{1}{|x-p|}\,+\,\frac{1}{|x-q|}\big)\,|v|,
$$
where $p,q$ are the intersection points with $\partial\Omega$ of the line at $x$ with tangent $v$. The Hilbert volume is defined
 for Borel sets $\mathcal{A}\subset \Omega$ by
 $$
 Vol_{\Omega}(\mathcal{A})\,=\,\int_{\mathcal{A}}\frac{\pi}{vol(B_x(1)))}\,dvol(x),
 $$ where $B_x(1)=\{v\in T_x\Omega\mid \|v\|<1\}$ and $dvol$ is the Lebesgue measure for the affine set.  The Hilbert metric and volume satisfy a {\it domain comparison principle} that is basic to our considerations.  For $\Omega\subseteq\Omega'$ and $x,y\in\Omega$, $\mathcal{A}\subset \Omega$ then
 $$
 d_{\Omega'}(x,y)\,\le d_{\Omega}(x,y)
 $$
 and
 $$ 
 Vol_{\Omega'}(\mathcal{A})\le Vol_{\Omega}(\mathcal{A}).
 $$
 
 \section{Inscribed triangles and quadrilaterals}
 Fock-Goncharov study the space of $\mathcal{P}^n_3$ of pairs of convex $n$-gons, one inscribed in the other, considered modulo the \pg\ action.  The space can be considered as a discrete approximation to the space of closed convex curves with the inner polygon representing the closed curve and outer polygon representing the tangent lines.  By identifying along sides, an inscribed $n$-gon is described as a gluing of inscribed triangles.  We study the parameters for gluing triangles. First we review the single parameter describing a labeled triple flag. See \cite[Fig. 1]{FGmcps} and Figure 1. 
 
 \begin{figure}
 	\centering
 	  \makebox[11cm][c]{\includegraphics[width=1.7\linewidth]{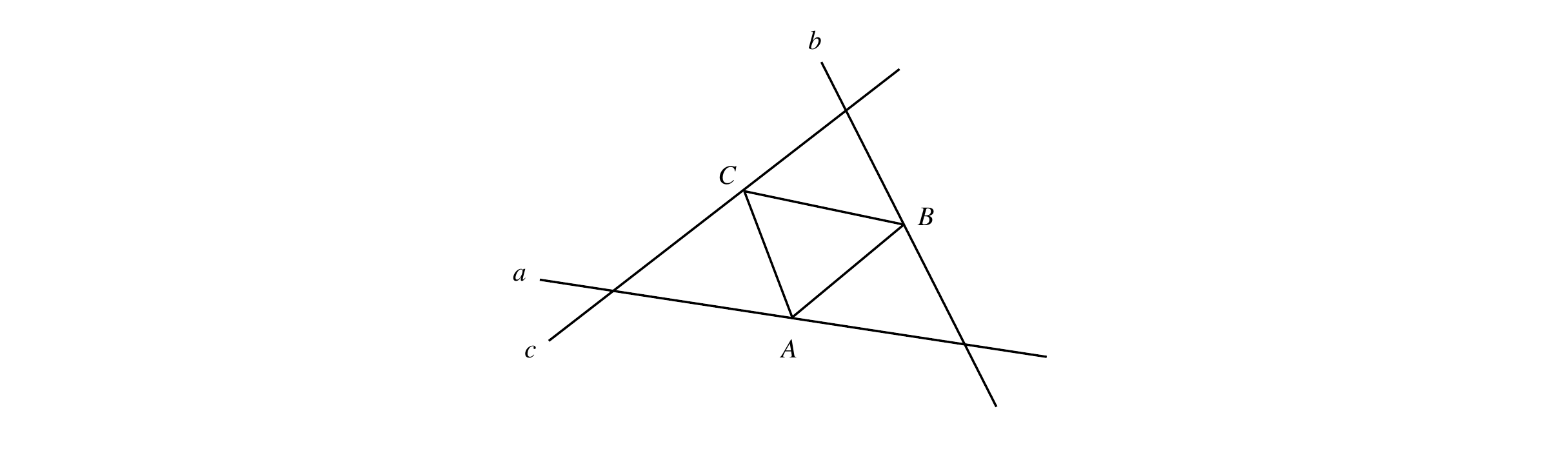}}
 	\caption{Inscribed triangles.}
 	\label{fig:triangle}
 \end{figure}

 Given a triple of flags $AaBbCc$ in general position, an element of \pg\ is selected to move the intersections of lines to the points representing the \R\ coordinate axes.  Diagonal matrices provide the stabilizer of the coordinate axes.  Two flag points can be further normalized.  The location of the third flag point characterizes the configuration modulo \pg.  Fock-Goncharov define the triple ratio invariant valued in $\mathbb{R}-\{0\}$ by the triple ratio of the flag.  The invariant is replaced by its reciprocal when the orientation of the flag is reversed.   The invariant is defined in an affine chart as follows.   The labeling orders the lines.  Write each flag point as an affine combination of the intersection points of lines: $A=(1-t_A)a\cap c+t_A a\cap b$, $B=(1-t_B)a\cap b+t_B b\cap c$ and $C=(1-t_C)b \cap c+t_C c\cap a$.  The $t$-parameters are in the interval $(0,1)$ precisely when in the affine chart the flag points are in between the intersection points.   
 \begin{proposition}\label{triag}
 The triple ratio invariant of $AaBbCc$ is $$
 \frac{t_A}{1-t_A}\frac{t_B}{1-t_B}\frac{t_C}{1-t_C}.
$$.  
 \end{proposition}  
\begin{proof}
	Use \pg\ to transform the flag so that for homogeneous coordinates $(\alpha,\beta,\gamma)$ and inhomogeneous coordinates $x=\alpha/\gamma, y=\beta/\gamma$,  we have $a$ is the line $y=0$, $b$ is the line $y=1-x$ and $c$ is the line $x=0$.  The vertices of $abc$ are $(0,0), (1,0)$ and $(0,1)$.  A section for the inhomogeneous coordinates is given by $(x,y,1)$.  The flag points are described as affine combinations as follows: $A$ is given as $(x_0,0)$ with $t_A=x_0$, $B$ is given as $(1-y_1,y_1)$ with $t_B=y_1$ and $C$ is given as $(0,y_2)$ with $t_C=1-y_2$.  The lines are described in homogeneous coordinates as $a$ is given as $(0,1,0)\cdot(\alpha,\beta,\gamma)=0$, $b$ is given as $(1,1,-1)\cdot(\alpha,\beta,\gamma)=0$ and $c$ is given as $(1,0,0)\cdot(\alpha,\beta,\gamma)=0$.  The triple ratio is
	\begin{multline*}
	\frac{(0,1,0)\cdot(1-t_B,t_B,1)(1,1,-1)\cdot(0,1-t_C,1)(1,0,0)\cdot(t_A,0,1)}{(0,1,0)\cdot(0,1-t_C,1)(1,1,-1)\cdot(t_A,0,1)(1,0,0)\cdot(1-t_B,t_B,1)}\\
	=\, \frac{t_A}{1-t_A}\frac{t_B}{1-t_B}\frac{t_C}{1-t_C}.
	\end{multline*}  
\end{proof}    
The above considerations are valid without the hypothesis that the flag points are between the intersection points.  General position provides that the flag points and intersection points are distinct.  Accordingly the triple ratio invariant is negative precisely when one or three flag points are not between the corresponding intersection points.  Fock-Goncharov \cite[Lemma 2.3]{FGmcps} note that the triple ratio invariant is negative if and only if the triangle $ABC$ is not inscribed in one of the four triangles determined by $abc$.

We are ready to study inscribed quadrilaterals and the Fock-Goncharov edge parameters for attaching triangles.  We start with a quadruple of flags  $AaBbCcDd$ in general position and select an element of \pg\ to normalize in inhomogeneous coordinates $(x,y)$,
\begin{equation}
	\label{norm} 
	A \mbox{ to } (0,0),\ B\mbox{ to } (1,1),\ C \mbox{ to } (0,2) \mbox{ and } a\cap c \mbox{ to } (-1,1).
\end{equation}  
See \cite[Fig. 2]{FGmcps} and Figure 2.  

\begin{figure}
	\includegraphics[width=2.2\linewidth]{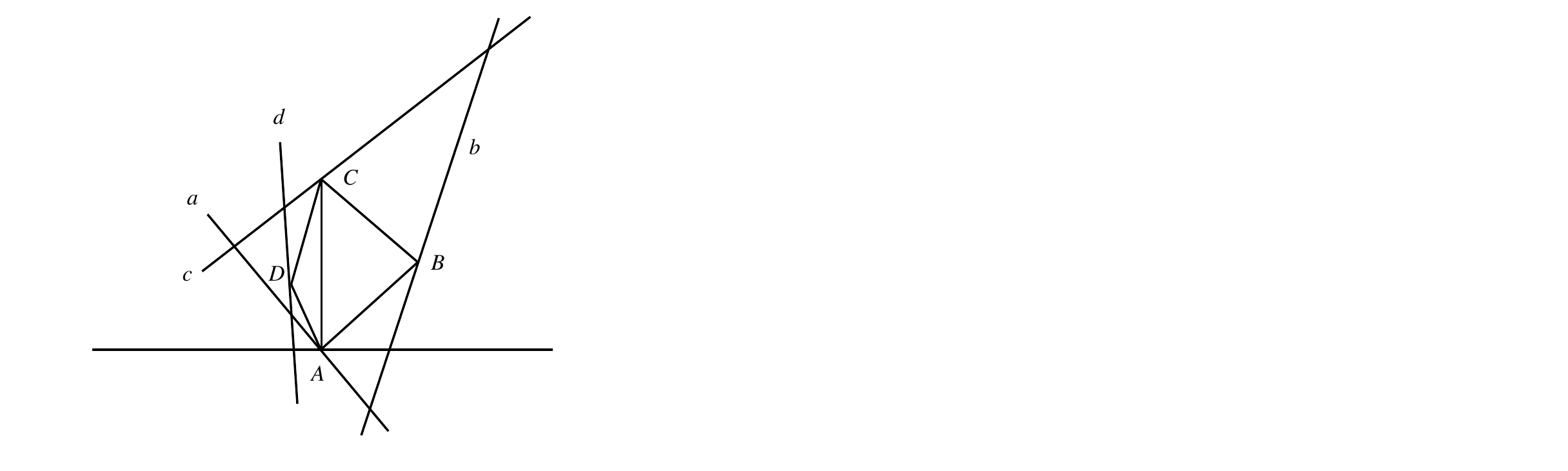}
	\caption[Inscribed quadrilaterals.]{An inscribed quadrilateral with normalization (\ref{norm}). }
	\label{fig:quadrilateral}
\end{figure}

\begin{proposition}
	\label{edgepar}
	With the Fock-Goncharov notation.  For the edge parameters $W,Z$,
	$$
	\mbox{the line AD has the equation}\;  y\,=\,(\frac{-2}{Z}-1)x
	$$
	and 
	$$
	\mbox{the line CD has the equation}\; y\,=\,(2W+1)x\,+\,2.
	$$
	The point $D$ has the coordinates $(\mu,\nu)$ with
	$$
	\mu\,=\,-F(V,W)\quad\mbox{and}\quad
		\nu\,=\,(1+2V)F(V,W),
	$$
	for 
	$$
	V\,=\,\frac{1}{Z}\quad \mbox{and}\quad F(V,W)\,=\,\frac{1}{W+V+1}.
	$$
	The inverse relations are
	$$
	W\,=\,\frac{2-\nu+\mu}{-2\mu}\quad\mbox{and}\quad Z=
	\frac{-2\mu}{\mu+\nu}.
	$$
	For $Y$ the triple ratio invariant of $ACD$ and $m$ the slope of the line $d$ then 
	$$
	Y\,=\,\frac{(2-\nu+\mu)(\nu-\mu m)}{(\nu+\mu)(-2+\nu-\mu m)}.
	$$
	The reflection in the line $y=1$ fixes $B$ and $a\cap c$, interchanges the flags $Aa$ and $Cc$, reverses orientation, replaces $\nu$ with $2-\nu$, replaces $m$ with $-m$, and replaces $W$ with $1/Z$, $Z$ with $1/W$, and $Y$ with $1/Y$.   The above formulas transform accordingly.  
\end{proposition} 
\begin{proof}
	Fock-Goncharov \cite[pg. 254]{FGmcps} define the {\it edge parameter} $Z$ by the cross ratio of the lines $a$, $AB$, $AC$, $AD$ in counterclockwise order.  The lines have the following equations and intersections with $y=1$: $a$ has equation $y=-x$ and intersection $-1$, $AB$ has equation $y=x$ and intersection $1$, $AC$ has equation $x=0$ and intersection $0$ and $AD$ has equation $y=\omega x$ and intersection $1/\omega$ for a parameter $-\infty<\omega<-1$.  The resulting cross ratio is $Z=-2/(\omega +1)$ and solving for $\omega$ gives $\omega=(-2/Z)-1$.   The equation $y=(-2/Z-1)x$ for the line $AD$ is established.  The reflection in the line $y=1$ reverses orientation.  The lines at $A$ in counterclockwise order reflect to lines at $C$ in clockwise order.  The edge parameter $W$ is defined by lines in counterclockwise order.  The change from clockwise to counterclockwise order (with the flag line listed first) is given by the permutation $1432$, resulting in a reciprocal for the cross ratio.  The reflection transforms $\omega$ to $-\omega$ and interchanges $A$ and $C$.  We find $W=(-\omega+1)/-2$, $\omega=2W+1$ and $y=(2W+1)x+2$ for the line $CD$.  The formulas for the intersection point $D$ of the lines $AD$ and $CD$ are found by simultaneous solving.  To find the inverse relation start with $\nu/\mu=-(1+2V)$ which gives the formula for $Z$ and then substitute into the equation $\mu=-F(V,W)$.  
	
	To continue we use homogeneous coordinates $(\alpha,\beta,\gamma)$, inhomogeneous coordinates $x=\alpha/\gamma$, $y=\beta/\gamma$ and the coordinate section $(x,y,1)$.  The  line $z$ has equation $y=-x$ and homogeneous equation $(1,1,0)\cdot(\alpha,\beta,\gamma)=0$; the line $c$ has equation $y=x+2$ and homogeneous equation $(1,-1,2)\cdot(\alpha,\beta,\gamma)=0$; the line $d$ has equation $y=m(x-\mu)+\nu$ and homogeneous equation $(m,-1,\nu-\mu m)\cdot(\alpha,\beta,\gamma)=0$. The points are as follows: $A$ is $(0,0,1)$, $C$ is $(0,2,1)$ and $D$ is $(\mu,\nu,1)$. The triple ratio invariant of $ACD$ is given by the triple ratio
	\begin{align*}
	Y&=\frac{(1,1,0)\cdot(0,2,1)(1,-1,2)\cdot(\mu,\nu,1)(m,-1,\nu-\mu m)\cdot(0,0,1)}
	{(1,1,0)\cdot(\mu,\nu,1)(1,-1,2)\cdot(0,0,1)(m,-1,\nu-\mu m)\cdot(0,2,1)}\\
	&=\frac{(2-\nu+\mu)(\nu-\mu m)}{(\nu+\mu)(-2+\nu-\mu m)}.
	\end{align*}
	 The transformation rules follow by inspection. The proof is complete.     
\end{proof}

We note by Proposition \ref{triag} that the invariant $Y$ is positive since $A$ and $C$ are outside of the triangle $adc$.  The point $D$ which is interior to the triangle $ACa\cap c$ is the moduli for the quadrilateral modulo the \pg\ action.  Equivalently the pair $(W,Z)\in(0,\infty)\times(0,\infty)$ is moduli for the quadrilateral. We find that the two parameterizations give rise to different completions for the moduli domain - differences at $A$ and $C$.  We apply the above the formulas and begin with the complement of neighborhoods of $A$ and $C$.  The point $a\cap c$ corresponds to the $(W,Z)$ limiting value $(0,\infty)$.  The open line segment $a\cap c\,A$ corresponds to the $(W,Z)$ limiting values in $(0,\infty)\times \infty$.  Similarly the open line segment $a\cap c\,C$ corresponds to the limiting values $0\times(0,\infty)$.  The open line segment $AC$ corresponds to $(W,Z)$ limiting to $(\infty ,0)$ and $WZ$ limiting to a value $(2-\nu)/\nu$ in $(0,\infty)$ - in particular $\nu=2/(WZ+1)$.  

In contrast at $A$ and $C$ the limiting $(W,Z)$ values depend on the direction of approach.  The limiting values describe the blow ups of the two points.  We consider the description at $A$; the transformation rules provide the description at $C$.  Consider that  $(\mu,\nu)$ is close to $A$ and $\nu=-s\mu$; in particular $s$ is the slope of the line $AD$ with $s\in(1,\infty)$.  Substituting $\nu=-s\mu$ we find $Z=2/(s-1)$ or $s=(2/Z)+1$; as earlier $Z$ parameterizes the slope of $AD$.  From Proposition \ref{edgepar}, $\nu=(1+2/Z)F(1/Z,W)$ and a point $(\mu,\nu)$ is close to $A$ if and only if $\nu$ is close to zero if and only if one of three cases occurs.  First, $W$, $Z$ and $Z/W$ tend to infinity; second, $W$ tends to infinity and $Z$ tends to a non zero value; third, $W$, $Z$ and $W/Z$ tend to infinity.  The distance from $A$ to $D$ is $(1+(1+2V)^2)^{1/2}F(V,W)\approx c'/W$, for a positive constant.   In overall summary, $Z$ describes the slope of $AD$ and $1/W$ is the parameter for the distance from $A$ to $D$.  For $D$ close to $A$ with $\nu=-s\mu$ then the triple ratio invariant $Y$ is close to the limiting value $(s+m)/(1-s)$; the line $d$ has limiting slope $(1-s)Y-s$.  In particular for the limiting value $s=1$ then $d$ has limiting slope $-1$ which coincides with the slope $a$.  Otherwise for $s>1$ the limiting slope of $d$ is strictly less than the slope of $a$ and a {\it corner} forms in the limit.  

\section{Estimating Hilbert area}

We analyze the Hilbert area of inscribed triangles and quadrilaterals.  We use the domain comparison principle and the bounds for the Hilbert area form for a rectangle to derive area bounds.  For a lower bound we use a larger rectangle containing the domain and for an upper bound we use a covering by overlapping rectangles contained within the domain.  Colbois-Vernicos-Verovic \cite[Proposition 6]{CVV} show for a rectangle $\{(x,y)\mid |x|<L, |y|<H\}$ that the Hilbert area element $dHilb$ is bounded as $1/4\, d\mathcal{V}\le dHilb\le 1/2\,d\mathcal{V}$ for
\begin{equation}
\label{dV}
d\mathcal{V}\,=\,\frac{\pi HLdxdy}{(L^2-x^2)(H^2-y^2)}.
\end{equation}
The reader may note the comparison to the hyperbolic metric $R|dz|/(R^2-|z|^2)$ for the radius $R$ disc.  Translating the domain by $L$, the one-dimensional form satisfies the elementary comparison
$$
\frac12\frac{dx}{x}\,\le\,\frac{Ldx}{x(2L-x)}\,\le\,\frac{1}{2(1-\delta)}\frac{dx}{x},
$$
for $0<x<2\delta L,\ 0<\delta <1$. To illustrate our estimation approach we begin by re deriving the existing area bound for an inscribed triangle.  

\begin{figure}
	\centering
	\includegraphics[width=1.0\linewidth]{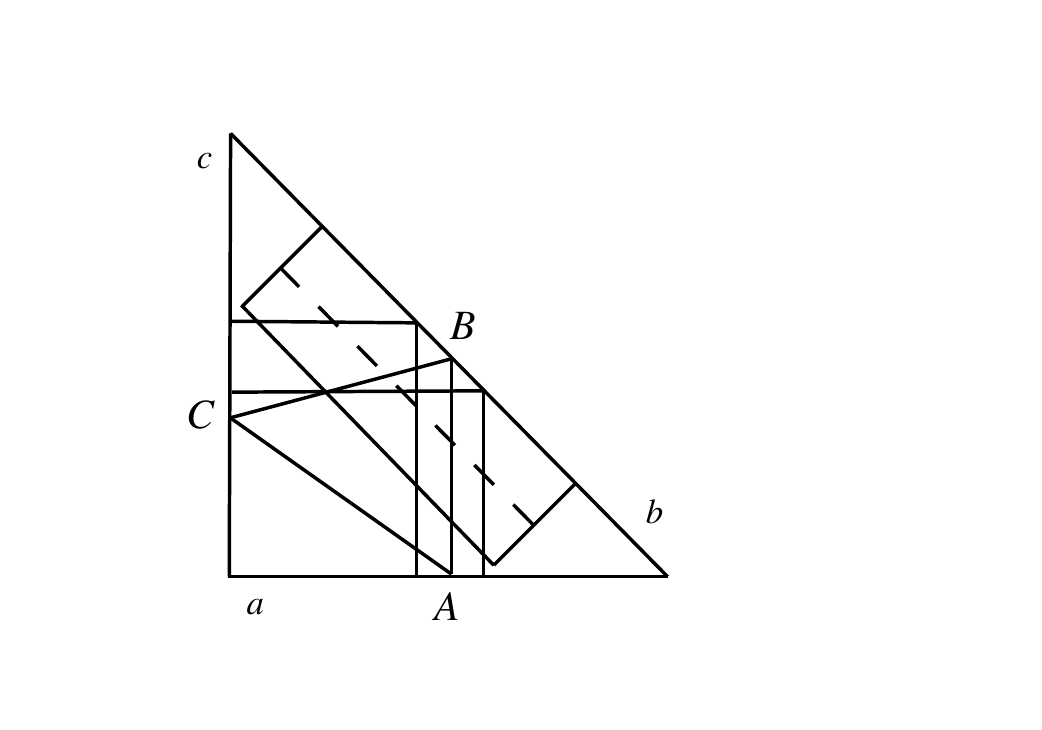}
	\caption{Diagram for the proof of Proposition \ref{tribound}.  The rectangles $R_+$, $R_-$ with edges along the axes and the rectangle $R_0$ with edge along the hypotenuse together cover the triangle $ABC$.  The midline $\mathcal{L}$ of $R_0$ is dotted.  }
	\label{fig:trianglecover}
\end{figure}

\begin{proposition}\label{tribound}
	There are positive constants $c_1$ and $c_2$, such that for an inscribed triangle $AaBbCc$ with invariant $T$, then the Hilbert area of $ABC$ with respect to $abc$ satisfies
	$$
	c_1(1\,+\,(\log T)^2)\,\le\,Vol_{abc}(ABC)\,\le\,c_2(1\,+\,(\log T)^2).
	$$
\end{proposition}
\begin{proof}
	As already noted, reversing the flag orientation replaces the triple ratio invariant with its reciprocal.  We now assume $T\le 1$.  Use \pg\ to transform the flags so that for inhomogeneous coordinates $(x,y)$, we have $a$ is the line $y=0$, $b$ is the line $x+y=2$, $c$ is the line $x=0$, $A$ is the point $(1,0)$ and $B$ is the point $(1,1)$.   Since $T\le 1$, $C$ is a point $(0,\tau)$ with $0<\tau\le 1$.   The triangle $ABC$ has upper side line $y=(1-\tau)x+\tau$, lower side line $y=-\tau x+\tau$ and vertical side $x=1$.
	
	We use the larger square $\{(x,y)\mid 0<x<2, 0<y<2\}$ to derive a lower bound for the $abc$ area form.   From (\ref{dV}), the comparison area element is bounded below by $\pi dxdy/16xy$.   The area lower bound is
	\begin{equation}\label{keyintegral}
	\int_0^1\int_{-\tau x+\tau}^{(1-\tau)x+\tau}\frac{dy}{y}\frac{dx}{x}\,=\,\int_0^1-\frac{1}{x}\log(1-x)dx\,+\,\int_0^1\frac{1}{x}\log((\frac{1}{\tau}-1)x+1)dx.  
	\end{equation}
The first integral on the right evaluates to $\pi^2/6$.  The second integrand on the right is analytic at $x=0$.  We can differentiate under the integral to find the $\tau$-derivative 
\begin{multline*}
-\frac{1}{\tau^2}\int_0^1\frac{dx}{(\frac{1}{\tau}-1)x+1}\\=-\frac{1}{\tau(1-\tau)}\log((\frac{1}{\tau}-1)x+1)\bigg\vert_0^1=\frac{1}{\tau(1-\tau)}\log\tau.  
\end{multline*}
By table lookup, a $\tau$-anti-derivative is $Li_2(1-\tau)\,+\,(\log \tau)^2/2$, for the second polylogarithm $Li_2$, a bounded function on the unit-interval.  The polylogarithm and integral $\int_0^1\frac{1}{x}\log((\frac{1}{\tau}-1)x+1)dx$ are both regular for $0< \tau\le1$ and both vanish for $\tau=1$.   We conclude that
$$
\int_0^1\int_{-\tau x+\tau}^{(1-\tau)x+\tau}\frac{dy}{y}\frac{dx}{x}\,=\,\frac{\pi^2}{6}\,+\,\frac{(\log\tau)^2}{2}\,+\,Li_2(1-\tau),
$$
the desired area lower bound.   

For the area upper bound we use a covering of $ABC$ by the three rectangles contained in $abc$.  See Figure \ref{fig:trianglecover}. We use the area elements of the rectangles to bound the area element of the triangle $abc$ on regions of $ABC$.  For a parameter $\delta>0$, the rectangle $R_+$ (resp. $R_-$) has edges along the $x$ and $y$ axes and vertex $(1+\delta,1-\delta)$ (resp. $(1-\delta,1+\delta)$).  The third rectangle $R_0$ has an edge along the line segment from $(0,2)$ to $(2,0)$ and has its dimensions adjusted such that the vertices $(1-\delta,1+\delta)$ and $(1+\delta,1-\delta)$ are interior to the $R_0$ rectangle edge and such that $(1-\delta,1-\delta)$ is interior to the rectangle.  We will make an additional adjustment to $\delta$ in the following. The possible triangles $ABC$ are all contained in the unit square with vertices $(0,0), (1,0), (1,1)$ and $(0,1)$.  Our approach will use bounds for neighborhoods of the vertices of $ABC$ by integrals over sectors $\{(x,y)\mid  -y\le mx \le y,0<y<h\}$, in particular bounding by
$$
\int_0^h\int_{y=-mx}^{y=mx}\,dx\frac{dy}{y}\,=\,\frac{2h}{m}.
$$
Since we seek upper bounds, we may enlarge integration regions to simplify considerations.  We will refer to such a bound as a {\it sector bound}.
We consider two cases for the vertex parameter: $\tau\le 1/4$ and $1/4\le \tau\le 1$.  

For $\tau\le 1/4$, we first use $R_0$ to estimate the area form in a neighborhood of $(1,1)$.  Consider a midline $\mathcal{L}$ of $R_0$, parallel to the edge along the line $(0,2)$ to $(2,0)$ with $\mathcal{L}$ bisecting $R_0$.  See Figure \ref{fig:trianglecover}. Consider the subregion of $ABC$ above the line $\mathcal{L}$.  
We now adjust $\delta$ to ensure that $(1-\delta,1-\delta)$ is above the line $\mathcal{L}$.  
The $R_0$ Hilbert area of the subregion is bounded by a sector bound.  Since $\tau\le 1/4$, the subregion of $ABC$ below $\mathcal{L}$ is now bounded away from the edges $x=1+\delta$ and $y=1-\delta$ of $R_+$.  It follows that on the subregion of $ABC$ below $\mathcal{L}$, the $R_+$ Hilbert area form is bounded above by a multiple of the $dxdy/xy$ area.   The Hilbert area is now bounded by a multiple of the above integral (\ref{keyintegral}), a suitable bound.   

For $1/4\le\tau\le 1$, we consider a neighborhood of each $ABC$ vertex separately.  The vertex at $(1,1)$ is considered above.  The condition $1/4\le\tau\le 1$ provides that for the sectors at $(1,0)$ and $(0,\tau)$ the slopes of the sides of $ABC$ are bounded away from the slopes of $abc$.   The sector bound provides a uniform bound for the Hilbert area of uniform neighborhoods of the vertices $(1,0)$ and $(0,\tau)$.  Finally for $1/4\le\tau\le 1$, the union of the triangles $ABC$ with neighborhoods of vertices removed forms a compact subset of the larger triangle $abc$. There is a uniform bound for the corresponding Hilbert areas.  The upper bound for Hilbert area is complete.     
\end{proof}
We are ready to present the main result on varying the edge parameters.  
\begin{theorem}\label{main}
	With the above notation for the convex quadrilateral $AaBbCcDd$ and normalization of Figure 2.  The quadrilateral is parameterized by the point $D$ within the triangle $ACa\cap c$.  The Hilbert area of $ABCD$ with respect to $abcd$ approaches infinity as $D$ approaches the boundary of the triangle with the possible exceptions of $D$ approaching $A$ with $Z$ approaching infinity and of $D$ approaching $C$ with $W$ approaching zero.  For $D$ with coordinates $(\mu,\nu)$, the Hilbert area is comparable to 
	$1+(1/Z)\log(-1/\mu)$ in the first instance and is comparable to $1+W\log(-1/\mu)$ in the second instance. The area bounds are uniform for the triple ratio invariants of $ABC$ and $ACD$ in compact subintervals of $(0,\infty)$.
\end{theorem}

We find that the limiting area for $D$ approaching $A$ or $C$ depends on microlocal considerations.  For $s=\nu/\mu$ (resp. $s=(2-\nu)/\mu$) the slope of the secant line $AD$ (resp. the secant line $CD$), then the edge parameter satisfies $Z=-2/(s+1)$ (resp. $1/W=-2/(s+1)$).  It will suffice to analyze the setting of $D$ approaching $A$ since the reflection in the line $y=1$ interchanges the two situations. 

\begin{proof}
	We proceed by cases considering the subset of the boundary that the point $D$ approaches. The first case is for the flag lines $a\cup c$ but not the points $A$ or $C$.  In this case an open segment of the line $AD$ (resp. $CD$) approaches an open segment on the flag $a$ (resp. $c$).  This is a sufficient condition for the area to approach infinity.  The next case is for the line segment $AC$ but not the points $A$ or $C$.  In this instance by convexity the flag line $d$ approaches vertical.  The line segments $AD$ and $CD$ approach the line $d$ and the area approaches infinity.  The next case is for $D$ approaching $A$ with the slope of the secant approaching negative infinity.  From Proposition \ref{edgepar} we consider the formula for the triple ratio invariant and proceed by contradiction.  If the slope $m$ of $d$ is bounded with $(\mu,\nu)$ tending to $(0,0)$, then $Y$ is comparable to $2(\nu-\mu m)/(\nu+\mu)(-2)=(m-\nu/\mu)/(1+\nu/\mu)$ which is negative for $\nu/\mu$ large - a contradiction of $Y$ positive. Consequently the flag line $d$ becomes vertical and approaches the line $AC$.  The area approaches infinity.  The next case is for $D$ approaching $A$ with the slope of the secant limiting to a value in the interval $(-\infty,-1)$.  The limiting configuration has the vertex $A$ at the vertex $a\cap d$, a degenerate configuration.  Let $s$ be the limiting value of $\nu/\mu$.  By Proposition \ref{edgepar} the triple ratio invariant is comparable to $2(s-m)/((s+1)(-2-\mu m))$. Since the triple ratio invariant is positive and $\mu$ is negative it follows that $m$ is bounded.   It now follows for the $D$ tends to $A$ limit that $Y=(m-s)/(s+1)$ or equivalently $m=s+(s+1)Y$, an increasing function of $s$.  Since $-\infty<s<-1$ it follows that $m<-1$ and in the limit there is a corner at $a\cap d$.  The limiting configuration has the vertices $D$ and $a\cap d$ at the vertex $A$. A vertex of the triangle $ABC$ coincides with a vertex of the quadrilateral $abcd$. See Figure \ref{fig:quadrilateral}.  Given $\delta>0$, the $abcd$ area forms converge uniformly on $\mathcal{N}=\{(x,y)\mid (x,y)\in ABC,\,\delta<d((x,y),A)<1\}$.  Provided the area of the limiting configuration is infinite in a neighborhood of $A$, it will follow that the area of $ABCD$ approaches infinity as $D$ approaches $A$. (Given $M>0$ large, there is a $\delta>0$ such that the limiting configuration area of $\mathcal{N}$ is greater than $M$.  For nearby configurations the area of $\mathcal{N}$ is then greater than $M$.)  Use a projective transformation to convert the limiting quadrilateral $abcd$ to a square with an inscribed triangle with a vertex at a corner of the square.  The area of the triangle in a neighborhood of the vertex is given by an integral 
	$$
	\int_0^1\int_{y=m_1x}^{y=m_2x}\frac{dy}{y}\frac{dx}{x}
	$$
	with $m_2>m_1$, where we have used the elementary comparison for the Hilbert area form at the corner of a square.   The area is infinite, the desired conclusion.
	
	The final case, the principal case, is for $D$ approaching $A$ with the slope of the secant line limiting to $-1$; the flag $d$ limits to the flag $a$. The analysis is provided in the following. The reader can check that all estimates are uniform for the triple ratio invariants in compact subintervals of $(0,\infty)$.
\end{proof}

\begin{figure}
	\centering
	\includegraphics[width=.9\linewidth]{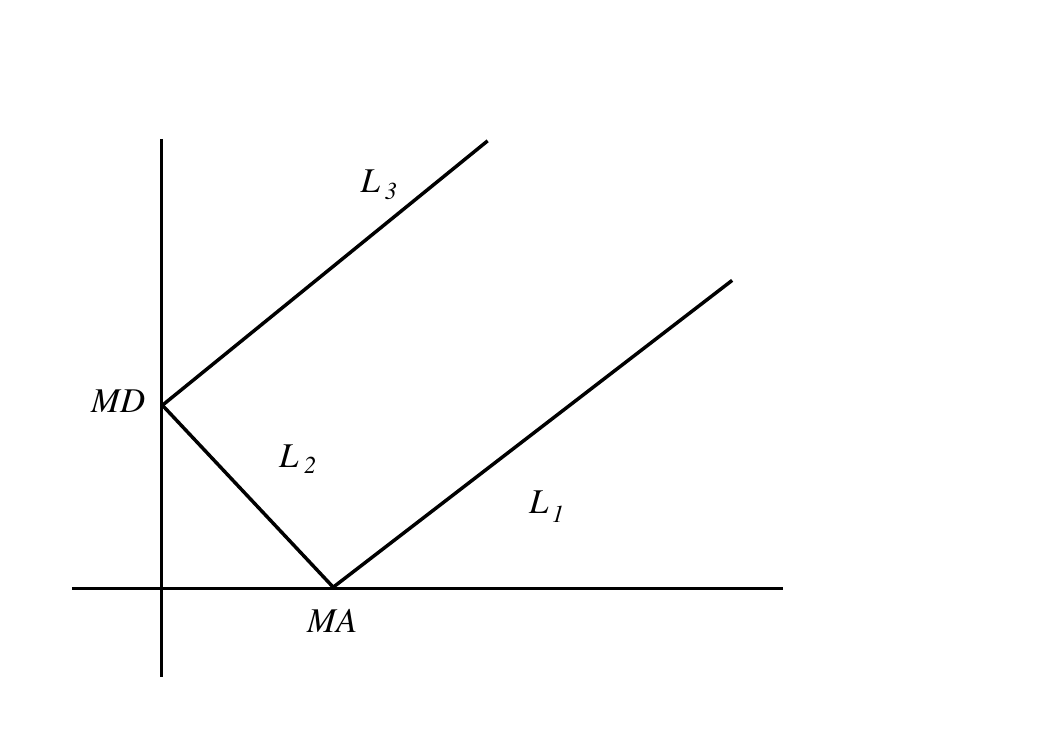}
	\caption[]{Diagram for the proof of Proposition \ref{logbhv}.  The $\mathcal{M}$ transform of a neighborhood of $A$.}
	\label{fig:mad}
\end{figure}

\begin{proposition}\label{logbhv}
	With the above notation.  For $D$ approaching $A$ with the slope of the secant line approaching $-1$ then the Hilbert area of $ABCD$ is comparable to $1+(1/Z)\log(-1/\mu)$.
\end{proposition}
\begin{proof}
	First we bound the area of $ABCD$ in the complement of a neighborhood of $A$.  For a neighborhood of $B$ (resp. $C$) we choose a square contained in $abcd$ with $B$ (resp. $C$) interior to an edge.  The sector bound from Proposition \ref{tribound} provides a bound for the area of neighborhoods of $B$ and $C$.   We show below for the slope of the secant line tending to $-1$ that $m$, the slope of $d$, also tends to $-1$.  It follows that the $ABCD$ complement of neighborhoods of $A, B$ and $C$ is a compact set separated from $abcd$.  The area of the complement is uniformly bounded.   
	
	We begin with formulas for the calculations. The setting is that $\mu,\nu$ and $1+s$ are tending to zero.  Write the triple ratio invariant formula as 
	$$
	(\mu+\nu)\tilde Y\,=\,\mu m-\nu \quad \mbox{where}\quad\tilde Y\,=\,Y\,\frac{(2-\nu+\mu m)}{(2-\nu+\mu)}.
	$$
Divide the triangle formula by $\mu$ and write $s=\nu/\mu$ for the slope of the $AD$ secant to find $(1+s)\tilde Y=m-s$ or equivalently $m=s+(1+s)\tilde Y$ or equivalently  $1+m=(1+s)(1+\tilde Y)$.  Our calculations involve the two quantities 
$$
\sigma\,=\,\frac{\mu+\nu}{1+m}\,=\,\mu\,\frac{1+s}{1+m}\,=\,\frac{\mu}{1+\tilde Y}
$$
and
$$ 
\rho\,=\,\frac{\nu-\mu m}{1+m}\,=\,\mu\,\frac{s-m}{1+m}\,=\,\frac{-\mu\tilde Y}{1+\tilde Y}.
$$

The first matter is the coordinates for the flag intersection point $a\cap d$. The $a$ flag has equation $y=-x$ and the $d$ flag has equation $y-\nu=m(x-\mu)$.  The abscissa equation is $-x-\nu=m(x-\mu)$.  It follows that the $a\cap d$ coordinates are $(-\rho,\rho)$ and that the intersection point tends to $A$. The main consideration is a change of coordinates to transform the flags $a$ and $d$ to the horizontal and vertical axes.   The Hilbert area element comparison element will be $dxdy/xy$.  The affine change of coordinates is translation by $(\rho,-\rho)$ and multiplication by 
$$
\mathcal{M}\,=\,\left(\begin{matrix} 
-m & 1
\\1 & 1
\end{matrix}
\right).
$$ 
The translation of $D$ is $\sigma(1,m)$ and the translation of $A$ is $\rho(1,-1)$, then we have
$$
\mathcal{M}D\,=\,\sigma(1+m)(0,1)^{\top}\quad\mbox{and}\quad\mathcal{M}A\,=\,-\rho(1+m)(1,0)^{\top}.
$$ 
We now consider $DC$ and $AB$ as giving displacements rather than locations.  In particular consider $DC$ and $AB$ as tangent vectors which transform by $\mathcal{M}$, giving
\begin{align*}
&\mathcal{M}DC\,=\,\mathcal{M}\left(\begin{matrix} -\mu\\2-\nu\end{matrix}\right)\,=\,\left(\begin{matrix} 2 +\mu m-\nu\\2-\nu-\mu\end{matrix}\right)\quad\mbox{and}\\
&\mathcal{M}AB\,=\,\mathcal{M}\left(\begin{matrix} 1\\1\end{matrix}\right)\,=\,\left(\begin{matrix}  1-m\\2\end{matrix}\right).
\end{align*}
The slopes of the displacements satisfy:
$\mathcal{M}AB$ has slope $m_1=2/(1-m)$ and  $\mathcal{M}DC$ has slope $m_2=(2-\nu-\mu)/(2+\mu m-\nu)$.  It now follows that the transform $L_1$ of the line from $A$ to $B$ has equation $y=m_1(x+\rho(1+m))$; the transform $L_2$ of the line from $A$ to $D$ has equation $y=(\sigma/\rho) x+\sigma(1+m)$; the transform $L_3$ of the line from $D$ to $C$ has equation $y=m_2 x+\sigma(1+m)$. See Figure \ref{fig:mad}. The lines provide three sides of $ABCD$. Neighborhoods $\mathcal{U}$ of $A$ in $ABCD$ are defined by  horizontal lines $\mathcal{L}$.  Viewed from $A$, the intersection of a line $\mathcal{L}$ with the triangle $ABC$ sweeps out an angle interval $[\pi/4,\pi/2]$. For $D$ very close to $A$, viewed from $a\cap d$ the intersection of $\mathcal{L}$ with $ABCD$ sweeps out an angle interval very close to $[\pi/4,\pi/2]$.  For vectors in this angle interval the length of a vector and its product with $\mathcal{M}$ have uniformly comparable length.  In consequence the transform of a neighborhood $\mathcal{U}$ is a bounded neighborhood of the origin in the image.  Since the flags $a$ and $d$ transform to the horizontal and vertical, the transformed quadrilateral $\mathcal{M}abcd$ contains the intersection of a neighborhood of the origin and the first quadrant, hence contains a square with vertex at the origin. See Figure 3.  By convexity the transformed quadrilateral is also contained in a square with vertex at the origin. By the one-dimensional comparison, it now follows that the Hilbert area element on $\mathcal{MU}$ is uniformly comparable to $dxdy/xy$.

Our considerations give that the Hilbert area of a neighborhood of $A$ in $ABCD$ is the sum of the two integrals (\ref{firstint}) and (\ref{secondint}).   The first integral is
\begin{equation}\label{firstint}
\int_0^{-\rho(1+m)}\int_{(\sigma/\rho) x+\sigma(1+m)}^{m_2x+\sigma(1+m)}\frac{dy}{y}\frac{dx}{x}\,=\,\int_0^{-\rho(1+m)}\frac{1}{x}\log\frac{m_2x+\sigma(1+m)}{(\sigma/\rho) x+\sigma(1+m)\,}dx,
\end{equation} 
where the logarithm vanishes to first order for $x=0$.  We substitute $x=\mu(1+m)v$ and factor $\mu(1+m)/(1+\tilde Y)$ out of the numerator and denominator for the logarithm.   The result is
$$
\int_0^{\tilde Y/(1+\tilde Y)}\frac{1}{v}\log\frac{m_2(1+\tilde Y)v+1}{-v(1+\tilde Y)/\tilde Y + 1}\,dv,
$$
a convergent integral, and since $\tilde Y$ tends to $Y$ and $m_2$ tends to $2$ as $(\mu,\nu)$ tends to $(0,0)$, the overall  contribution of the integral is comparable to unity. The second integral corresponds to the region bounded by the vertical line $x=-\rho(1+m)$, the lines $L_1$ and $L_3$, and a vertical line $x=c_*$ defining the neighborhood $\mathcal{U}$ of $A$.  The second integral is
\begin{equation}\label{secondint} 
\int_{-\rho(1+m)}^{c_*}\frac{1}{x}\log\frac{m_2x+\sigma(1+m)}{m_1(x+\rho(1+m))}\,dx.
\end{equation}     
Substitute $x=\mu(1+m)v$ to find the integral 
$$
\int_{\tilde Y/(1+\tilde Y)}^{c_*/(\mu(1+m))}\frac{1}{v}\log\frac{m_2(1+\tilde Y)v+1}{m_1((1+\tilde Y)v-\tilde Y)}\,dv
$$
after multiplying the logarithm numerator and denominator by the factor $(1+\tilde Y)$.  We rewrite the logarithm as
$$
\log\frac{m_2}{m_1}\,+\,\log\frac{(1+\tilde Y)v+1/m_2}{(1+\tilde Y)v-\tilde Y}.
$$
The second logarithm is bounded as $O(1/v)$ for $v$ large.  This term multiples the factor $1/v$ and so contributes $1/v^2$ to the integrand - a quantity integrable on $(1,\infty)$ - the overall contribution to the integral is comparable to unity independent of $\mu(1+m)$.  The final calculation is for $\log(m_2/m_1)$; the one-term expansion suffices. We have that 
\begin{multline*}
\frac{m_2}{m_1}\,=\,\frac{2-\nu-\mu}{2+\mu m-\nu}\frac{1-m}{2}\,=\,\frac{2-\mu(1+s)}{2+\mu(m-s)}\frac{2-(1+m)}{2}\\ =1-\frac12\mu(1+s)-\frac12\mu(m-s)-\frac12(1+m)+\dots\,=\,1-\frac12(1+\mu)(1+m)+\dots
\end{multline*}
and that $\log(m_2/m_1)$ is comparable to $-(1+m)$.  The resulting contribution to the integral is $-(1+m)\log(1/(\mu(1+m))$ with $-(1+m)$ comparable to $-(1+s)$ which by Proposition \ref{edgepar} is comparable to $1/Z$.  The expression is simplified upon noting that $v\log 1/v$ is bounded for $v$ small.  The proof is complete.  
\end{proof}

\section{Discussion.}

We discuss the applications considerations of Theorem \ref{main}.  First we present the central quantity in terms of Fock-Goncharov parameters and Cartesian coordinates.  From Proposition \ref{edgepar}, we have equivalent expressions 
\begin{multline*}
Q\,=\,\frac{1}{Z}\log\frac{-1}{\mu}\,=\,\frac{1}{Z}\log(1+W+1/Z)\\ =\,\frac{\mu+\nu}{-2\mu}\log\frac{-1}{\mu}\,=\,\frac{-(1+s)}{2}\log\frac{-1}{\mu},
\end{multline*}
for $(\mu,\nu),\,-\mu<\nu$, the Cartesian coordinates for the vertex $D$ and $s=\nu/\mu$ the slope of the secant line.  Since in application $Z$ is tending to infinity, the second expression is comparable to $(1/Z)\log(1+W)$.   We see that Hilbert area tending to infinity corresponds to $W$ tending to infinity exponentially faster than $Z$ tending to infinity or $\mu$ tending to zero exponentially faster than $1+s$ tending to zero.   
We now consider that $D$ is a point on a graph and study the property of the graph corresponding to a given Hilbert limiting area.  To study a graph $y=f(x)$ in the first quadrant we reflect about the vertical axis and set $x=-\mu$ and $y=\nu$.  The central quantity is now
$$
Q_f(x)\,=\,\frac{f(x)-x}{x}\log\frac{1}{x}.
$$
The context for Theorem \ref{main} provides that $f(0)=0$, $f'(0)=1$ and $f(x)\ge x$.  In particular the slope of the secant line (before reflection) is $-1-(2/Z)$ with $Z$ tending to infinity. We find that the limiting Hilbert area is described by a rescaling of the graph - a second blow up. The following is immediate.  

\begin{proposition}\label{arealimit}
	Notation as above. For functions $y=f(x)$ with $f(0)=0$, $f'(0)=1$ and $f(x)\ge x$, then $Q_f$ is increasing in $f$.  For $g(x),\,x>0,$ a non negative function with limit at zero and $f(x)=x(1+(g(x)/\log(1/x)))$ then $
	\lim_{x\rightarrow 0}Q_f(x)\,=\,\lim_{x\rightarrow 0}g(x).
	$
\end{proposition} 

The analysis of Proposition \ref{logbhv} is local at the coalescing vertices and applies equally for convergence of adjacent flags of an inscribed $n$-gon as follows.  Observe that the configuration is characterized by the convergence of the adjacent flags $d$ and $a$.  See Figure \ref{fig:quadrilateral}.  That the polygon is inscribed implies that the slope of the secant line approaches $-1$. The ensuing analysis applies for a pair of converging adjacent flags of an $n$-gon.    

In the setting of a discrete subgroup of \pg\  acting properly on a strictly convex domain in \rp\  the boundary is $C^{1+\epsilon}$, \cite{BenCdI}.  A $C^{1+\epsilon}$ function normalized as above has an expansion $f(x)=x+O(x^{1+\epsilon})$. From the proposition we have that $\lim_{x\rightarrow 0}Q_f(x)=0$. For a vertex tracing a  $C^{1+\epsilon}$ curve, the Hilbert area of a quadrilateral inscribed in a quadrilateral is a bounded function. 

Goldman introduced the \emph{twisting} and \emph{bulging} deformations for convex projective structures as generalizations of the Fenchel-Nielsen twist deformation for hyperbolic structures, \cite{Gdprj}. Bonahon-Kim studied the relation between the Goldman parameters and the Fock-Goncharov parameters.  We now use  Theorem \ref{main} to study Kim's postulation that bounding Hilbert area bounds bulging parameters.  

We begin with the effect of twisting and bulging on the Fock-Goncharov parameters.  Bonahon-Kim present the formulas in terms of logarithms of Fock-Goncharov parameters, \cite[Lemma 5.1]{BoKi}.  We state the formulas directly.  The $u$-twist variation of $(Z,W)$ is given as $(e^uZ,e^uW)$ and the $v$-bulge variation is given as $(e^{-3v}Z,e^{3v}W)$.   The matrix $\begin{pmatrix}
	1 & -3 \\ 1 & 3
\end{pmatrix}$
of exponent coefficients is invertible - fixing a reference point $(Z_0,W_0)$, the twist-bulge parameters $(u,v)$  parameterize attaching triangles. The twist (resp. bulge) deformation fixes the ratio $Z/W$ (resp. the product $ZW$) and varies the product (resp. the ratio).  We consider that the bulge deformation diverges - the ratio tends to infinity or zero.  No restriction is placed on the twist deformation; the values $Z,W$ are only restricted by the behavior of their ratio.  We introduce notation to catalog the possible cases.  We write respectively $0,+$ or $\infty$, if a quantity limits respectively to zero, a positive number or to infinity.   We begin with the condition that $Z/W$ limits to infinity.   In lexicographic order, the five possible cases for $(Z,W)$ are $(0,0), (+,0), (\infty,0), (\infty, +)$ and $(\infty,\infty)$.  For all cases we have that $Z\gg W$.  We refer to Proposition \ref{edgepar} for the equations of the flag lines with parameters $Z$ and $W$; we refer to Figure 2 for the diagram.  The behaviors for the vertex $D$ are as follows: for the case $(0,0)$, $D$ limits to $C$; for the cases $(+,0),(\infty,0)$ and $(\infty,+)$, $D$ limits to $a\cup c-A-C$; for the case $(\infty,\infty)$, $D$ limits to $A$. In the first case the Hilbert area may or may not be bounded depending whether $Z$ is exponentially larger than $W$.  For the combination of the second, third, and fourth cases, the Hilbert area tends to infinity and the possible limits fill out $a\cup c-A-C$ since the values $Z,W$ are only constrained by the condition that the ratio tends to infinity.  In the fifth case $(\infty,\infty)$ the central quantity $Q$ is bounded since $Z\gg W$.  We are ready to consider that $Z/W$ limits to zero.  The five possible cases for $(Z,W)$ are $(0,0),(0,+),(0,\infty),(+,\infty)$ and $(\infty,\infty)$.  For all cases $Z\ll W$. The behaviors for $D$ are as follows: for the case $(0,0)$, $D$ limits to $C$; for the cases $(0,+),(0,\infty)$ and $(+,\infty)$, $D$ limits to $AC-A-C$; for the case $(\infty,\infty)$, $D$ limits to $A$. In the first and fifth cases the Hilbert area may or may not be bounded depending whether $Z$ is exponentially smaller than $W$. For the second, third and fourth cases the Hilbert area tends to infinity and the possible limits fill out $AC-A-C$ since the values $Z,W$ are only constrained by the condition that the ratio tends to zero.  Overall we have that $(Z,W)$ can tend to $(0,0)$ or $(\infty,\infty)$ with $D$ tending to $A$ or $C$ with Hilbert area bounded; in all other cases Hilbert area is unbounded.  In particular for $(Z,W)$ tending to $(0,0)$ or $(\infty,\infty)$ with $Z/W$ tending to $0$ or $\infty$ and $W$ bounded between $Z^{1-\epsilon},Z^{1+\epsilon}$, then the Hilbert area of quadrilaterals inscribed in quadrilaterals is bounded. 

\begin{proposition} Notation as above for inscribed  quadrilaterals.  Bounds for the $abcd$ Hilbert area of $ABCD$ do not provide bounds for the bulging parameters. 	
\end{proposition}

The observation has a consequence for strictly convex domains.   Given an $n$-gon inscribed in an $n$-gon, the outer $n$-gon can be \emph{smoothed} to a strictly convex domain, exponentially close to the original $n$-gon in neighborhoods of the inscribed vertices.  The Hilbert area varies continuously under such smoothing and the edge parameters are unchanged. Start with an inscribed quadrilateral and consider a sequence of deformations with bulging parameters diverging and with uniformly bounded Hilbert areas.  Smooth the outer quadrilaterals, keeping the Hilbert areas uniformly bounded.  The result is a sequence of strictly convex domains with inscribed quadrilaterals with uniformly bounded Hilbert areas and diverging bulging parameters.  The domains are not specified as bulgings of a single domain. 

\begin{proposition}
	There are sequences of strictly convex smooth domains with inscribed quadrilaterals with uniformly bounded Hilbert areas and diverging bulging parameters. 
\end{proposition} 

We now compare to the considerations of Proposition 4.2 in \cite{Kideg} and Proposition 3.5 in \cite{SunZ}. In the proofs of the propositions it is presented (with the present notation) that for $Z/W$ tending to infinity, $D$ can limit only to $a\cap c$ and for $Z/W$ tending to zero, $D$ can limit to any point of $AC$.  We find for $Z/W$ diverging that the point $D$ can always limit to $A$ and $C$. Our analysis of Theorem \ref{main}, Proposition \ref{logbhv} and the  above sequence example indicates that the limit behavior of the Hilbert area for $D$ approaching $A$ or $C$ is at least a delicate matter.  Kim has communicated to the author that he has updated his result to: provided that the triple ratio and twist parameters are bounded, then the Hilbert area bounds bulging parameters.                  


\providecommand\WlName[1]{#1}\providecommand\WpName[1]{#1}\providecommand\Wl{Wlf}\providecommand\Wp{Wlp}\def\cprime{$'$}

\end{document}